\documentclass[nosumlimits,twoside]{amsart}

\usepackage{amsfonts, amsmath, amssymb}
\usepackage[bookmarksnumbered,plainpages,hypertex]{hyperref}
\usepackage{graphicx}
\usepackage{float}
\usepackage{srcltx}
\usepackage[all]{xy}

\newtheorem{theorem}{\sc Theorem}[section]
\newtheorem{proposition}[theorem]{\sc Proposition}
\newtheorem{notation}[theorem]{\sc Notation}

\newtheorem{lemma}[theorem]{\sc Lemma}
\newtheorem{corollary}[theorem]{\sc Corollary}
\theoremstyle{definition}
\newtheorem{definition}[theorem]{\sc Definition}

\theoremstyle{remark}

\newcommand{\thlabel}[1]{\label{th:#1}}
\newcommand{\thref}[1]{Theorem~\ref{th:#1}}
\newcommand{\prlabel}[1]{\label{pr:#1}}
\newcommand{\prref}[1]{Proposition~\ref{pr:#1}}
\newcommand{\lelabel}[1]{\label{le:#1}}
\newcommand{\leref}[1]{Lemma~\ref{le:#1}}
\newcommand{\colabel}[1]{\label{co:#1}}
\newcommand{\coref}[1]{Corollary~\ref{co:#1}}
\newcommand{\delabel}[1]{\label{de:#1}}
\newcommand{\deref}[1]{Definition~\ref{de:#1}}
\newcommand{\eqlabel}[1]{\label{eq:#1}}
\newcommand{\equref}[1]{(\ref{eq:#1})}

\begin{document}
\title{Quasi-bialgebra Structures and Torsion-free Abelian Groups}
\dedicatory{Dedicated to Toma Albu and Constantin N\u{a}st\u{a}sescu}

\author{Alessandro Ardizzoni}
\address{University of Turin, Department of Mathematics "G. Peano", via
Carlo Alberto 10, I-10123 Torino, Italy}
\email{alessandro.ardizzoni@unito.it}
\urladdr{www.unito.it/persone/alessandro.ardizzoni}

\author{Daniel Bulacu}
\address{Faculty of Mathematics and Informatics, University
of Bucharest, Str. Academiei 14, RO-010014 Bucharest 1, Romania}
\email{daniel.bulacu@fmi.unibuc.ro}
\urladdr{http://fmi.unibuc.ro/ro/departamente/matematica/bulacu\_daniel/}

\author{Claudia Menini}
\address{University of Ferrara, Department of Mathematics, Via Machiavelli
35, Ferrara, I-44121, Italy}
\email{men@unife.it}
\urladdr{http://www.unife.it/utenti/claudia.menini}

\subjclass[2010]{Primary 16W30; Secondaries 18D10; 16S34}
\thanks{This paper was written while the first and the third authors were members of GNSAGA.
The first author was partially supported by the research grant ``Progetti di
Eccellenza 2011/2012'' from the ``Fondazione Cassa di Risparmio di Padova e
Rovigo''. The second author was supported by
the strategic grant POSDRU/89/1.5/S/58852, Project
``Postdoctoral program for training scientific researchers" cofinanced
by the European Social Fund within the Sectorial Operational Program Human
Resources Development 2007 - 2013.}

\begin{abstract}
We describe all the quasi-bialgebra structures of a group algebra over a
torsion-free abelian group. They all come out to be triangular in a unique way.
Moreover, up to an isomorphism, these quasi-bialgebra
structures produce only one (braided) monoidal structure on the category of their
representations. Applying these results to the algebra of Laurent polynomials,
we recover two braided monoidal categories introduced in \cite{CG} by S.
Caenepeel and I. Goyvaerts in connection with Hom-structures (Lie algebras,
algebras, coalgebras, Hopf algebras).
\end{abstract}

\keywords{Quasi-bialgebras, Hom-category, Laurent polynomials, torsion-free abelian groups}
\maketitle

\section*{Introduction}

Let $\mathcal{C}$ be a category. In \cite[Section 1]{CG},
S. Caenepeel and I. Goyvaerts introduce the so called Hom-category
$\mathcal{H}\left( \mathcal{C}\right)$
in order to investigate Hom-structures (Lie algebras, algebras,
coalgebras, Hopf algebras) from the monoidal categorical point of view.
More exactly, if $\mathcal{C}$ is the category of modules over a commutative ring, then
$\mathcal{H}\left( \mathcal{C}\right) $ admits a symmetric monoidal structure
with respect to which (co)algebras in $\mathcal{H}\left( \mathcal{C}\right) $
coincide with Hom-(co)algebras, Hopf algebras with Hom-Hopf algebras and Lie algebras with Hom-Lie algebras,
respectively.

Now, fix a field $\Bbbk $ and denote by $\mathfrak{M}$ the category of $%
\Bbbk $-vector spaces. The category $\mathcal{H}\left( \mathfrak{M}\right) $
has objects pairs $\left( V,f_{V}\right) $ with $V\in \mathfrak{M}$ and $%
f_{V}\in \mathrm{Aut}_{\Bbbk }(V).$ A module over the polynomial ring $\Bbbk %
\left[ X\right] $ is a $\Bbbk $-vector space $V$ together with an element $%
g_{V}\in \mathrm{End}_{\Bbbk }\left( M\right) .$ In order to have $g_{V}$
invertible, as in the case of $\mathcal{H}\left( \mathfrak{M}\right) $, the ring $%
\Bbbk \left[ X\right] $ must be replaced with the algebra of Laurent
polynomials $\Bbbk \left[ X,X^{-1}\right] $ or, equivalently, with the
group algebra $\Bbbk \left[ \mathbb{Z}\right] $. These facts suggest a
connection between the category $\mathcal{H}(\mathfrak{M})$ and the category
of $\Bbbk \left[ \mathbb{Z}\right] $-modules. They are actually isomorphic, and
this will be proved in \prref{LaurentHom}.

As a matter of fact, it was proved in \cite{CG} that the category
$\mathcal{H}\left( \mathfrak{M}\right)$ has two different braided monoidal
structures, denoted by $\mathcal{H}(\mathfrak{M})$ and
$\widetilde{\mathcal{H}}(\mathfrak{M})$, respectively. This leads us to consider braided monoidal
structures on the category of $\Bbbk \left[ \mathbb{Z}\right] $-modules. We restrict ourselves
to the case when these structures are induced by the strict monoidal structure of
$\mathfrak{M}$. It comes out that in this case we have to compute the quasi-bialgebra
structures of the group algebra $\Bbbk[\mathbb{Z}]$, cf. \thref{reconstructionQuasi}.
We will do this in a wider context, by replacing $\mathbb{Z}$ with a torsion-free abelian group.
In fact, our arguments are valid for any abelian group $G$ with the property that, for any
natural number $n\in\{1,2,3\}$, the group of units of the group algebra $\Bbbk[G^n]$ is trivial
(that is, any invertible element of $\Bbbk[G^n]$ is a nonzero scalar multiple of an element
in $G^n$), where $G^n$ stands for the direct product of $n$ copies of $G$. But, if
this is the case then by \cite[Lemma 1.1]{Passman} we have that $G$ is torsion-free
(note that the exceptions listed in the Lemma have $|G|=2$ and hence they fulfill the requirement only for $n=1$).
For the other way around, if $G$ is a torsion-free abelian group then by \cite[Corollary 2.5]{Gilmer-Heitmann}, inductively,
it follows that $\Bbbk \left[ G^n\right]$ has the group of units trivial. (Note that this can be obtained also
from \cite[Lemmas 1.6, 1.7 \& 1.9(ii)]{Passman}).

Now, for a torsion-free abelian group we show that the third Harrison
cohomology group $H^3_{\rm Harr}(\Bbbk[G], \Bbbk, \mathbb{G}_m)$ is trivial
(\prref{HarrisonTorsionFree}). Moreover, any Harrison $3$-cocycle on $\Bbbk[G]$ is
uniquely determined by a pair $(h, g)$ of elements of $G$, and this allows us to
describe, up to an isomorphism, all the quasi-bialgebra structures on the group algebra
$\Bbbk[G]$. When we specialize this for the
multiplicative cyclic group $\left\langle g\right\rangle \cong
\mathbb{Z}$ we obtain that the quasi-bialgebra structures on the group algebra
$\Bbbk [\left\langle g\right\rangle]$ are, up to an isomorphism,
completely determined by triples
$(q, a, b)\in (\Bbbk\backslash\{0\})\times \mathbb{Z}\times \mathbb{Z}$, see
\thref{teo:LaurentQuasi}. Furthermore, all of them are deformations of
the ordinary bialgebra structure of $\Bbbk [\left\langle g\right\rangle]$ by an invertible
element in $\Bbbk[\left\langle g\right\rangle] \otimes \Bbbk[\left\langle g\right\rangle]$,
and so, up to an isomorphism, the category ${}_{\Bbbk[\langle g\rangle]}\mathfrak{M}$ admits a
unique (strict) monoidal structure. The same is valid for the braided situation, and this is
mostly because the ordinary bialgebra $\Bbbk[\langle g\rangle]$ has a unique quasi-triangular
(actually triangular) structure (\coref{BraidedZ}).

As we have already explained, the categories $\mathcal{H}(\mathfrak{M})$ and
${}_{\Bbbk[\langle g\rangle]}\mathfrak{M}$ are isomorphic. Consequently,
we have a one to one correspondence between the (braided) monoidal structures on
$\mathcal{H}(\mathfrak{M})$ and the (braided) monoidal structures on
${}_{\Bbbk[\langle g\rangle]}\mathfrak{M}$. In Theorem \ref{teo:mainVec} we endow
$\mathcal{H}\left( \mathfrak{M}\right)$ with the symmetric monoidal category
structures induced by those of ${}_{\Bbbk[\langle g\rangle]}\mathfrak{M}$ that
we previously computed. In general,
such a structure depends on a triple $(q, a, b)\in (\Bbbk\backslash\{0\})\times \mathbb{Z}\times \mathbb{Z}$,
and this is why we denoted it by $\mathcal{H}^{a, b}_p\left( \mathfrak{M}\right)$. We have
isomorphisms of symmetric monoidal categories,
${_{\Bbbk[\left\langle g\right\rangle]}}\mathfrak{M}\cong
{_{\Bbbk[\left\langle g\right\rangle]_{q}^{a,b}}}\mathfrak{M}\cong \mathcal{H}%
_{q}^{a,b}\left( \mathfrak{M}\right)$, see Corollary \ref{coro:Hequiv}.
Since $\mathcal{H}_{1}^{0,0}\left( \mathfrak{M}\right) $ and $%
\mathcal{H}_{1}^{1,-1}\left( \mathfrak{M}\right) $ can be identified, as symmetric
monoidal categories, with
$\mathcal{H}(\mathfrak{M})$ and $\widetilde{\mathcal{H}}(\mathfrak{M})$, respectively,
we obtain that $\mathcal{H}(\mathfrak{M})$ and $\widetilde{\mathcal{H}}(\mathfrak{M})$
are isomorphic as symmetric monoidal categories
(Proposition \ref{pro:Hrev} and \coref{SGiso}). We recover in this way
\cite[Proposition 1.7]{CG}, in the particular case when the base category is $\mathfrak{M}$.

\section{Preliminaries}

\label{preliminares}

In this section, we shall fix some basic notation and terminology.

\begin{notation}
Throughout this paper $\Bbbk $ will denote a field. All vector spaces will
be defined over $\Bbbk $. The unadorned tensor product $\otimes $ will
denote the tensor product over $\Bbbk $ if not stated otherwise. The
category of vector spaces will be denoted by $\mathfrak{M}$.
\end{notation}

\textbf{Monoidal Categories.}(\cite[Chap. XI]{Kassel}) Recall that a
\emph{monoidal category}\textbf{\ }is a category $\mathcal{M}$ endowed with
an object $\mathbf{1}\in \mathcal{M}$\textbf{\ } (called \emph{unit}), a
functor $\otimes :\mathcal{M}\times \mathcal{M}\rightarrow \mathcal{M}$
(called \emph{tensor product}), and functorial isomorphisms $%
a_{X,Y,Z}:(X\otimes Y)\otimes Z\rightarrow $ $X\otimes (Y\otimes Z)$, $l_{X}:%
\mathbf{1}\otimes X\rightarrow X,$ $r_{X}:X\otimes \mathbf{1}\rightarrow X,$
for every $X,Y,Z$ in $\mathcal{M}$. The functorial morphism $a$ is called
the \emph{associativity constraint }and\emph{\ }satisfies the \emph{Pentagon
Axiom, }that is the following relation
\begin{equation*}
(U\otimes a_{V,W,X})\circ a_{U,V\otimes W,X}\circ (a_{U,V,W}\otimes
X)=a_{U,V,W\otimes X}\circ a_{U\otimes V,W,X}
\end{equation*}%
holds true, for every $U,V,W,X$ in $\mathcal{M}.$ The morphisms $l$ and $r$
are called the \emph{unit constraints} and they obey the \emph{Triangle
Axiom, }that is $(V\otimes l_{W})\circ a_{V,\mathbf{1},W}=r_{V}\otimes W$,
for every $V,W$ in $\mathcal{M}$.

A \emph{monoidal functor}\label{MonFun} (also called strong monoidal in the
literature)
\begin{equation*}
(F,\phi _{0},\phi _{2}):(\mathcal{M},\otimes ,\mathbf{1},a,l,r\mathbf{%
)\rightarrow (}\mathcal{M}^{\prime }\mathfrak{,}\otimes ^{\prime },\mathbf{1}%
^{\prime },a^{\prime },l^{\prime },r^{\prime }\mathbf{)}
\end{equation*}%
between two monoidal categories consists of a functor $F:\mathcal{M}%
\rightarrow \mathcal{M}^{\prime },$ an isomorphism $\phi
_{2}(U,V):F(U)\otimes ^{\prime }F(V)\rightarrow F(U\otimes V),$ natural in $%
U,V\in \mathcal{M}$, and an isomorphism $\phi _{0}:\mathbf{1}^{\prime
}\rightarrow F(\mathbf{1})$ such that the diagram
\begin{equation*}
\xymatrixcolsep{63pt}\xymatrixrowsep{35pt}\xymatrix{ (F(U)\otimes'
F(V))\otimes' F(W) \ar[d]|{a'_{F(U),F(V),F(W)}} \ar[r]^{\phi_2(U,V)\otimes'
F(W)} & F(U\otimes V)\otimes' F(W) \ar[r]^{\phi_2(U\otimes V,W)} &
F((U\otimes V)\otimes W) \ar[d]|{F(a_{ U,V, W})} \\ F(U)\otimes'
(F(V)\otimes' F(W)) \ar[r]^{F(U)\otimes' \phi_2(V,W)} & F(U)\otimes'
F(V\otimes W) \ar[r]^{\phi_2(U,V\otimes W)} & F(U\otimes (V\otimes W)) }
\end{equation*}%
is commutative, and the following conditions are satisfied:
\begin{equation*}
{F(l_{U})}\circ {\phi _{2}(\mathbf{1},U)}\circ ({\phi _{0}\otimes }^{\prime }%
{F(U)})={l}^{\prime }{_{F(U)}},\text{\quad }{F(r_{U})}\circ {\phi _{2}(U,%
\mathbf{1})}\circ ({F(U)\otimes }^{\prime }{\phi _{0}})={r}^{\prime }{%
_{F(U)}.}
\end{equation*}

The monoidal functor is called \emph{strict }if the isomorphisms $\phi
_{0},\phi _{2}$ are identities of $\mathcal{M}^{\prime }$. A \emph{braided
monoidal category} $(\mathcal{M},c)$ is a monoidal category $(\mathcal{M}%
,\otimes ,\mathbf{1})$ equipped with a \emph{braiding} $c$, that is an
isomorphism $c_{U,V}:U\otimes V\rightarrow V\otimes U$, natural in $U,V\in
\mathcal{M}$, satisfying, for all $U,V,W\in \mathcal{M}$,
\begin{eqnarray*}
a_{V,W,U}\circ c_{U,V\otimes W}\circ a_{U,V,W} &=&(V\otimes c_{U,W})\circ
a_{V,U,W}\circ (c_{U,V}\otimes W), \\
a_{W,U,V}^{-1}\circ c_{U\otimes V,W}\circ a_{U,V,W}^{-1} &=&(c_{U,W}\otimes
V)\circ a_{U,W,V}^{-1}\circ (U\otimes c_{V,W}).
\end{eqnarray*}%
Such a category is called \emph{symmetric} if we further have $c_{V,U}\circ
c_{U,V}=\mathrm{Id}_{U\otimes V}$ for every $U,V\in \mathcal{M}$.

A \emph{(symmetric) braided monoidal functor} is a monoidal functor $F:%
\mathcal{M}\rightarrow \mathcal{M}^{\prime }$ such that
\begin{equation*}
F\left( c_{U,V}\right) \circ \phi _{2}(U,V)=\phi _{2}(V,U)\circ c_{F\left(
U\right) ,F\left( V\right) }^{\prime }.
\end{equation*}%
More details on these topics can be found in \cite[Chapter XIII]{Kassel}%
.\medskip

\textbf{Quasi-Bialgebras.} The following definition is not the original one
given in \cite[page 1421]{Drinfeld-QuasiHopf}. We adopt the more general
form of \cite[Remark 1, page 1423]{Drinfeld-QuasiHopf} (see also \cite[%
Proposition XV.1.2]{Kassel}) in order to comprise the case of Hom-categories.

\begin{definition}\delabel{quasi-bialgebra}
A \emph{quasi-bialgebra} is a datum $\left( H,m,u,\Delta ,\varepsilon ,\phi
,\lambda ,\rho \right) $ where

\begin{itemize}
\item $\left( H,m,u\right) $ is an associative algebra;

\item $\Delta :H\rightarrow H\otimes H$ and $\varepsilon :H\rightarrow \Bbbk
$ are algebra maps;

\item $\lambda ,\rho \in H$ are invertible elements;

\item $\phi \in H\otimes H\otimes H$ is a counital $3$-cocycle i.e. it is an
invertible element and satisfies%
\begin{eqnarray*}
\left( H\otimes H\otimes \Delta \right) \left( \phi \right) \cdot \left(
\Delta \otimes H\otimes H\right) \left( \phi \right) &=&\left( 1_{H}\otimes
\phi \right) \cdot \left( H\otimes \Delta \otimes H\right) \left( \phi
\right) \cdot \left( \phi \otimes 1_{H}\right) , \\
\left( H\otimes \varepsilon \otimes H\right) \left( \phi \right) &=&\rho
\otimes \lambda ^{-1};
\end{eqnarray*}

\item $\Delta $ is quasi-coassociative and counitary i.e. it satisfies%
\begin{eqnarray*}
\left( H\otimes \Delta \right) \Delta \left( h\right) &=&\phi \cdot \left(
\Delta \otimes H\right) \Delta \left( h\right) \cdot \phi ^{-1}, \\
\left( \varepsilon \otimes H\right) \Delta \left( h\right) &=&\lambda
^{-1}h\lambda , \\
\left( H\otimes \varepsilon \right) \Delta \left( h\right) &=&\rho
^{-1}h\rho .
\end{eqnarray*}
\end{itemize}

A morphism of quasi-bialgebras (see \cite[page 371]{Kassel})
\begin{equation*}
\Xi :\left( H,m,u,\Delta ,\varepsilon ,\phi ,\lambda ,\rho \right)
\rightarrow \left( H^{\prime },m^{\prime },u^{\prime },\Delta ^{\prime
},\varepsilon ^{\prime },\phi ^{\prime },\lambda ^{\prime },\rho ^{\prime
}\right)
\end{equation*}%
is an algebra homomorphism $\Xi :\left( H,m,u\right) \rightarrow \left(
H^{\prime },m^{\prime },u^{\prime }\right) $ such that
\begin{eqnarray*}
(\Xi \otimes \Xi )\Delta &=&\Delta ^{\prime }\Xi ,\qquad
\varepsilon ^{\prime }\Xi =\varepsilon ,\qquad \left( \Xi \otimes
\Xi \otimes \Xi \right) \left( \phi \right) =\phi ^{\prime }, \\
\Xi \left( \lambda \right) &=&\lambda ^{\prime },\qquad \Xi \left(
\rho \right) =\rho ^{\prime }.
\end{eqnarray*}%
It is an isomorphism of quasi-bialgebras if, in addition, it is invertible.

We will use the following standard notation%
\begin{equation*}
\phi ^{1}\otimes \phi ^{2}\otimes \phi ^{3}:=\phi \text{ (summation
understood).}
\end{equation*}

In the case when $\phi$ is not trivial (that is, a nonzero scalar multiple of $1_H\otimes 1_H\otimes 1_H$)
and $\lambda=\rho=1_H$ we call $H$ an \emph{ordinary quasi-bialgebra}. If
$\phi$ is trivial and $\lambda=\rho=1_H$ we then land at the classical concept
of bialgebra.
\end{definition}

The definition of a quasi-bialgebra is based on the formalism of monoidal categories.
More exactly, if $H$ is a $\Bbbk $-algebra and $\Delta : H\rightarrow H\otimes H$ and
$\varepsilon : H\rightarrow \Bbbk $ are two algebra morphisms then the category of
left $H$-representations, ${}_H\mathfrak{M}$, endowed with the tensor product defined by $\Delta$
and with unit object $\Bbbk $ considered as left $H$-module via $\varepsilon$ is
monoidal if and only if $H$ is a quasi-bialgebra. As we will see later on, the existence of the
above two morphisms $\Delta$, $\varepsilon$ is not needed, as it is implied by the fact that the
monoidal structure on $\mathfrak{M}$ restricts to a monoidal structure on
${}_H\mathfrak{M}$. For further use, at this moment recall only the monoidal structure on
${}_H\mathfrak{M}$ produced by a quasi-bialgebra $H$.

Let $(H,m,u,\Delta ,\varepsilon ,\phi ,\lambda ,\rho )$ be a
quasi-bialgebra. It is well-known, see \cite[page 285 and Proposition XV.1.2]%
{Kassel}, that the category ${_{H}}\mathfrak{M}$ becomes
a monoidal category via the following structure. Given a left $H$-module $V$, we denote by $%
\mu =\mu _{V}^{l}:H\otimes V\rightarrow V,\mu (h\otimes v)=hv$, its left $H$%
-action. The tensor product of two left $H$-modules $V$ and $W$ is a module
via diagonal action i.e. $h\left( v\otimes w\right) =h_{1}v\otimes h_{2}w.$
The unit is $\Bbbk ,$ which is regarded as a left $H$-module via the trivial
action i.e. $h\kappa =\varepsilon \left( h\right) \kappa$, for all $h\in H$ and $\kappa\in \Bbbk$.
The associativity and unit constraints are defined, for all $V,W,Z\in {_{H}}\mathfrak{M}$ and $%
v\in V,w\in W,z\in Z,$ by%
\begin{gather*}
a_{V,W,Z}((v\otimes w)\otimes z):=\phi ^{1}v\otimes (\phi ^{2}w\otimes \phi
^{3}z), \\
l_{V}(1\otimes v):=\lambda v\qquad \text{and}\qquad r_{V}(v\otimes 1):=\rho
v.
\end{gather*}%
The monoidal category we have just described will be denoted by $({_{H}}%
\mathfrak{M},\otimes ,\Bbbk ,a,l,r).$

Let $\left( H,m,u,\Delta ,\varepsilon ,\phi ,\lambda ,\rho \right) $ be a
quasi-bialgebra. Given an invertible element $\alpha \in H\otimes H,$ we can
construct a new quasi-bialgebra $H_{\alpha }=\left( H,m,u,\Delta _{\alpha
},\varepsilon ,\phi _{\alpha },\lambda _{\alpha },\rho _{\alpha }\right) $
where
\begin{eqnarray*}
\Delta _{\alpha }\left( h\right) &=&\alpha \cdot \Delta \left( h\right)
\cdot \alpha ^{-1}, \\
\phi _{\alpha } &=&\left( 1_{H}\otimes \alpha \right) \cdot \left( H\otimes
\Delta \right) \left( \alpha \right) \cdot \phi \cdot \left( \Delta \otimes
H\right) \left( \alpha ^{-1}\right) \cdot \left( \alpha ^{-1}\otimes
1_{H}\right) , \\
\lambda _{\alpha } &=&\lambda \cdot \left( \varepsilon _{H}\otimes H\right)
\left( \alpha ^{-1}\right) , \\
\rho _{\alpha } &=&\rho \cdot \left( H\otimes \varepsilon _{H}\right) \left(
\alpha ^{-1}\right) .
\end{eqnarray*}%
By \cite[Lemma XV.3.4]{Kassel}, for every invertible element $\alpha \in
H\otimes H,$ the identity functor $\mathrm{Id}:{_{H}}\mathfrak{M}\rightarrow
{_{H_{\alpha }}}\mathfrak{M}$ induces a monoidal category isomorphism $%
\Gamma \left( \alpha \right) =\left( \mathrm{Id},\alpha _{0},\alpha
_{2}\right) :{_{H}}\mathfrak{M}\rightarrow {_{H_{\alpha }}}\mathfrak{M,}$
where $\alpha _{0}=\mathrm{Id}_{\Bbbk }$ and $\alpha _{2}\left( V,W\right)
\left( v\otimes w\right) :=\alpha ^{-1}\left( v\otimes w\right) $. The
inverse is $\Gamma \left( \alpha ^{-1}\right) .$

Notice also that for a quasi-bialgebra $H$ there is always an invertible element
$u\in H\otimes H$ such that $H_u$ is an ordinary quasi-bialgebra, and so
${}_H\mathfrak{M}$ is always monoidal isomorphic to a category for which
the unit object and the left and right
unit constraints are those of $\mathfrak{M}$ (see \cite{Drinfeld-QuasiHopf} or the proof
of \prref{ForgetfulMonoidal} below).

\begin{definition}
We refer to \cite[Proposition XV.2.2]{Kassel} but with a different
terminology (cf. \cite[page 1439]{Drinfeld-QuasiHopf}). A quasi-bialgebra $%
\left( H,m,u,\Delta ,\varepsilon ,\phi ,\lambda ,\rho \right) $ is called
\emph{quasi-triangular} whenever there exists an invertible element $R\in
H\otimes H$ such that, for every $h\in H,$ one has%
\begin{eqnarray}
\left( \Delta \otimes H\right) \left( R\right) &=&\left[
\begin{array}{c}
\left( \phi ^{2}\otimes \phi ^{3}\otimes \phi ^{1}\right) \left(
R^{1}\otimes 1\otimes R^{2}\right) \left( \phi ^{1}\otimes \phi ^{3}\otimes
\phi ^{2}\right) ^{-1} \\
\left( 1\otimes R^{1}\otimes R^{2}\right) \left( \phi ^{1}\otimes \phi
^{2}\otimes \phi ^{3}\right)%
\end{array}%
\right]  \label{form:quasi1} \\
\left( H\otimes \Delta \right) \left( R\right) &=&\left[
\begin{array}{c}
\left( \phi ^{3}\otimes \phi ^{1}\otimes \phi ^{2}\right) ^{-1}\left(
R^{1}\otimes 1\otimes R^{2}\right) \left( \phi ^{2}\otimes \phi ^{1}\otimes
\phi ^{3}\right) \\
\left( R^{1}\otimes R^{2}\otimes 1\right) \left( \phi ^{1}\otimes \phi
^{2}\otimes \phi ^{3}\right) ^{-1}%
\end{array}%
\right]  \label{form:quasi2} \\
\Delta ^{cop}\left( h\right) &=&R\Delta \left( h\right) R^{-1}
\label{form:quasi3}
\end{eqnarray}%
where $\phi :=\phi ^{1}\otimes \phi ^{2}\otimes \phi ^{3},$ $R=R^{1}\otimes
R^{2}.$

A \emph{morphism of quasi-triangular quasi-bialgebras} is a morphism $%
\Xi :H\rightarrow H^{\prime }$ of quasi-bialgebras such that $\left(\Xi\otimes
\Xi\right) \left( R\right) =R^{\prime }.$
\end{definition}

By \cite[Proposition XV.2.2]{Kassel}, ${_{H}}\mathfrak{M}=({_{H}}\mathfrak{M}%
,\otimes ,\Bbbk ,a,l,r)$ is braided if and only if there is an invertible
element $R\in H\otimes H$ such that $\left( H,m,u,\Delta ,\varepsilon ,\phi
,\lambda ,\rho ,R\right) $ is quasi-triangular. Note that the braiding is
given, for all $X,Y\in {_{H}}\mathfrak{M},$ by%
\begin{equation*}
c_{X,Y}:X\otimes Y\rightarrow Y\otimes X:x\otimes y\mapsto R^{2}y\otimes
R^{1}x.
\end{equation*}%
Moreover ${_{H}}\mathfrak{M}$ is symmetric if and only if we further assume%
\begin{equation}
R^{2}\otimes R^{1}=R^{-1}.  \label{form:triang}
\end{equation}%
Such a quasi-bialgebra will be called a \emph{triangular quasi-bialgebra}. A
morphism of triangular quasi-bialgebras is just a morphism of the underlying
quasi-triangular quasi-bialgebras structures.

Given an invertible element $\alpha \in H\otimes H,$ if $H$ is
(quasi-)triangular then so is $H_{\alpha }$ with respect to
\begin{equation*}
R_{\alpha }=\left( \alpha ^{2}\otimes \alpha ^{1}\right) R\alpha ^{-1}
\end{equation*}%
where $\alpha :=\alpha ^{1}\otimes \alpha ^{2}.$ This depends on the fact
that the monoidal category isomorphism $\Gamma \left( \alpha \right) =\left(
\mathrm{Id},\alpha _{0},\alpha _{2}\right) :{_{H}}\mathfrak{M}\rightarrow {%
_{H_{\alpha }}}\mathfrak{M}$ induces a (symmetric) braided structure on ${%
_{H_{\alpha }}}\mathfrak{M.}$ In particular we have%
\begin{equation*}
c_{H,H}^{H_{\alpha }}=\alpha _{2}^{-1}\left( H,H\right) \circ F\left(
c_{H,H}^{H}\right) \circ \alpha _{2}\left( H,H\right) .
\end{equation*}%
By \cite[Proposition XV.2.2]{Kassel}, $c_{H,H}^{H_{\alpha }}$ is of the form
$c_{H,H}^{H_{\alpha }}\left( x\otimes y\right) =R_{\alpha }^{2}y\otimes
R_{\alpha }^{1}x$ for all $x,y\in H$ where%
\begin{eqnarray*}
R_{\alpha }^{2}\otimes R_{\alpha }^{1} &=&c_{H,H}^{H_{\alpha }}\left(
1_{H}\otimes 1_{H}\right) =\left( \alpha _{2}^{-1}\left( H,H\right) \circ
F\left( c_{H,H}^{H}\right) \circ \alpha _{2}\left( H,H\right) \right) \left(
1_{H}\otimes 1_{H}\right) \\
&=&\left( \alpha _{2}^{-1}\left( H,H\right) \circ F\left( c_{H,H}^{H}\right)
\right) \left( \alpha ^{-1}\left( 1_{H}\otimes 1_{H}\right) \right) =\left(
\alpha _{2}^{-1}\left( H,H\right) \circ F\left( c_{H,H}^{H}\right) \right)
\left( \alpha ^{-1}\right) \\
&=&\alpha _{2}^{-1}\left( H,H\right) \left( R^{2}\left( \alpha ^{-1}\right)
^{2}\otimes R^{1}\left( \alpha ^{-1}\right) ^{1}\right) \\
&=&\alpha \left( R^{2}\left( \alpha ^{-1}\right) ^{2}\otimes R^{1}\left(
\alpha ^{-1}\right) ^{1}\right) =\alpha ^{1}R^{2}\left( \alpha ^{-1}\right)
^{2}\otimes \alpha ^{2}R^{1}\left( \alpha ^{-1}\right) ^{1}
\end{eqnarray*}%
and hence $R_{\alpha }=R_{\alpha }^{1}\otimes R_{\alpha }^{2}=\alpha
^{2}R^{1}\left( \alpha ^{-1}\right) ^{1}\otimes \alpha ^{1}R^{2}\left(
\alpha ^{-1}\right) ^{2}=\left( \alpha ^{2}\otimes \alpha ^{1}\right)
R\alpha ^{-1}$ as claimed above.

\section{Quasi-triangular quasi-bialgebra structures on a group algebra of a torsion-free abelian group}

Let $\Bbbk$ be a field and $G$ a torsion-free abelian group. We are interested in classifying the
(braided) monoidal structures on the category of left modules over the group algebra $\Bbbk[G]$ induced
by that of $\mathfrak{M}$. We shall see that this is equivalent to
the classification, up to deformation by an invertible element, of (quasi-triangular)
quasi-bialgebra structures on the group algebra $\Bbbk[G]$.

We start with the monoidal case. In general, for $H$ a $\Bbbk$-algebra, we say that the monoidal
structure on $\mathfrak{M}$ restricts to a monoidal structure on ${}_H\mathfrak{M}$ if
\begin{itemize}
\item[(a)] for any two left $H$-modules $X, Y$ the tensor product $X\otimes Y$ in $\mathfrak{M}$
admits a left $H$-module structure;
\item[(b)] the tensor product in $\mathfrak{M}$ of two left $H$-module morphisms is a morphism
in ${}_H\mathfrak{M}$, and so $\otimes$ induces a functor from ${}_H\mathfrak{M}\times {}_H\mathfrak{M}$
to ${}_H\mathfrak{M}$;
\item[(c)] $\Bbbk$, as a trivial $\Bbbk$-module, admits a left $H$-module structure;
\item[(d)] there exist functorial isomorphisms
$a=\left(a_{X, Y, Z}: (X\otimes Y)\otimes Z\rightarrow X\otimes (Y\otimes Z)\right)_{X, Y, Z\in {{}_H\mathfrak{M}}}$,
$l=\left(l_X: \Bbbk\otimes X\rightarrow X\right)_{X\in {{}_H\mathfrak{M}}}$ and
$r=\left(r_X: X\otimes \Bbbk\rightarrow X\right)_{X\in {{}_H\mathfrak{M}}}$ in ${}_H\mathfrak{M}$ such that the
Pentagon axiom and the Triangle Axiom are satisfied.
\end{itemize}

The next result is a slightly improved version of \cite[Proposition XV.1.2]{Kassel}
and can be viewed as a reconstruction type theorem for quasi-bialgebras.

\begin{theorem}\thlabel{reconstructionQuasi}
Let $\Bbbk$ be a field and $H$ a $\Bbbk$-algebra. Then there exists a one to one correspondence
between
\begin{itemize}
\item[$\bullet$] monoidal structures on ${}_H\mathfrak{M}$ induced by the strict monoidal
structure of $\mathfrak{M}$;
\item[$\bullet$] quasi-bialgebra structures on $H$.
\end{itemize}
\end{theorem}

\begin{proof}
Assume that the strict monoidal structure on $\mathfrak{M}$ induces a monoidal
structure on ${}_H\mathfrak{M}$. In particular, this implies that we have a left $H$-module structure
$\cdot: H\otimes (H\otimes H)\rightarrow H\otimes H$ on $H\otimes H$.
If we define $\Delta: H\rightarrow H\otimes H$ given by $\Delta(h)=h\cdot (1_H\otimes 1_H)$, for all $h\in H$, we then
claim that $\Delta$ is an algebra map. Indeed, it is clear that $\Delta(1_H)=1_H\otimes 1_H$. To see that $\Delta$ is
multiplicative we proceed as follows. Let $X\in {}_H\mathfrak{M}$ and fix $x\in X$. Then
$\varphi_x: H\ni h\mapsto hx\in X$ is a left $H$-module morphism. Similarly, for $Y\in {}_H\mathfrak{M}$
and $y\in Y$ define $\varphi_y: H\rightarrow Y$, a left $H$-module morphism. According to (a) and (b)
we have $\varphi_x\otimes \varphi_y: H\otimes H\rightarrow X\otimes Y$ a left $H$-linear morphism, hence
\[
(\varphi_x\otimes \varphi_y)(h\cdot (h'\otimes h{''}))=h\cdot (h'x\otimes h{''}y),
\]
for all $h, h', h{''}\in H$. If we take $h'=h{''}=1_H$ and denote $\Delta(h)=h_1\otimes h_2$ (summation implicitly understood),
we then get that $h\cdot (x\otimes y)=h_1x\otimes h_2y$, for all $h\in H$, and so $\Delta$ determines completely the left $H$-module
structure on the tensor product $X\otimes Y$. It follows now easily that $\Delta$ is multiplicative,
providing that $H\otimes H$ has the usual componentwise algebra structure.

We look now at the condition (c). We show that giving a left $H$-module structure on $\Bbbk$ is equivalent to giving an algebra map
$\varepsilon: H\rightarrow \Bbbk$. Indeed, let $\cdot : H\otimes \Bbbk\rightarrow \Bbbk$ be a left $H$-module
structure on $\Bbbk$. Since $\cdot$ is $\Bbbk$-linear we have $h\cdot \kappa=h\cdot(\kappa 1_\Bbbk)=
\kappa (h\cdot 1_\Bbbk)=(\kappa h)\cdot 1_\Bbbk$, for all $\kappa\in \Bbbk$ and $h\in H$. So if we define
$\varepsilon: H\rightarrow \Bbbk$ given by $\varepsilon(h):=h\cdot 1_\Bbbk$, for all $h\in H$,
then $\kappa\varepsilon(h)=\varepsilon(\kappa h)$, for all $h\in H$ and $\kappa\in \Bbbk$.
Otherwise stated, $\varepsilon$ is $\Bbbk$-linear and
$h\cdot \kappa=(\kappa h)\cdot 1_\Bbbk=\varepsilon(\kappa h)=\kappa \varepsilon(h)$, for
all $\kappa\in \Bbbk$ and $h\in H$. Now, $\varepsilon$ is multiplicative since, for all $h, g\in H$,
$$
\varepsilon (hg)=(hg)\cdot 1_\Bbbk=h\cdot (g\cdot 1_\Bbbk)=h\cdot \varepsilon(g)=\varepsilon(\varepsilon(g)h)=
\varepsilon(h)\varepsilon(g).
$$
In addition, $\varepsilon(1_H)=1_H\cdot 1_\Bbbk=1_\Bbbk$, and therefore $\varepsilon$ is an algebra morphism.

Conversely, if $\varepsilon: H\rightarrow \Bbbk$ is an algebra
map then clearly $\Bbbk$ is a left $H$-module via the structure defined by
$h\cdot \kappa=\varepsilon(h)\kappa$, for all $h\in H$ and $\kappa\in \Bbbk$. It is immediate that
the two correspondences defined above are inverses of each other. As far as we are concerned,
retain that (c) implies the existence of an algebra map $\varepsilon: H\rightarrow \Bbbk$ such that
$h\cdot \kappa=\varepsilon(h)\kappa$, for all $h\in H$ and $\kappa\in \Bbbk$.

Concluding, the desired monoidal structure on ${}_H\mathfrak{M}$ is induced by a
triple $(H, \Delta, \varepsilon)$ as in the statement of \cite[Proposition XV.1.2]{Kassel}. Consequently,
$(H, \Delta, \varepsilon, \phi, \lambda, \rho)$ is a quasi-bialgebra, where
\[
\phi=a_{H, H, H}(1_H\otimes 1_H\otimes 1_H)~,~\lambda=l_H(1_\Bbbk\otimes 1_H)~\mbox{\rm and}~\rho=r_H(1_H\otimes 1_\Bbbk)~,
\]
respectively. For the other way around see the comments made after \deref{quasi-bialgebra}.
\end{proof}

If we restrict ourself to the strict monoidal case we then get the following well-known result.

\begin{corollary}\colabel{reconstructionHopf}
Let $H$ be a $\Bbbk$-algebra. Then there is a one to one correspondence between
\begin{itemize}
\item[$\bullet$] the strict monoidal structures on ${}_H\mathfrak{M}$ such that the forgetful functor
$\mathfrak{U}: {}_H\mathfrak{M}\rightarrow \mathfrak{M}$ is strict monoidal;
\item[$\bullet$] bialgebra structures on $H$.
\end{itemize}
\end{corollary}

If $H$ is a quasi-bialgebra then the forgetful functor $\mathfrak{U}: {}_H\mathfrak{M}\rightarrow \mathfrak{M}$
is not necessarily monoidal, although the monoidal structure on ${}_H\mathfrak{M}$ is induced by the strict
monoidal structure on $\mathfrak{M}$. More exactly, we have the following situation.

\begin{proposition}\prlabel{ForgetfulMonoidal}
Let $(H, m, u, \Delta, \varepsilon, \phi, \lambda, \rho)$ be a quasi-bialgebra. Then the forgetful
functor $\mathfrak{U}: {}_H\mathfrak{M}\rightarrow \mathfrak{M}$ is monoidal if and only if there
exists an invertible element $\mathfrak{f}\in H\otimes H$ such that
$H_\mathfrak{f}$ is an ordinary bialgebra.
\end{proposition}
\begin{proof}
Let us start by noting that, without loss of generality, we can assume $\lambda=\rho=1_H$. This observation
is due to Drinfeld \cite{Drinfeld-QuasiHopf} and reduces the study of quasi-bialgebras to those of this type.

Indeed, applying $\varepsilon\otimes \varepsilon$ to the both sides of
$(H\otimes \varepsilon\otimes H)(\phi)=\rho\otimes \lambda^{-1}$
we get $\varepsilon(\phi^1)\varepsilon(\phi^2)\varepsilon(\phi^3)=\varepsilon(\rho)\varepsilon(\lambda^{-1})$.
On the other hand, it follows from $\varepsilon(h_1)h_2=\lambda^{-1}h\lambda$ that $\varepsilon(h_1)\varepsilon(h_2)=\varepsilon(h)$,
for any $h\in H$. Therefore, applying $\varepsilon \otimes \varepsilon \otimes \varepsilon \otimes \varepsilon$ to
the both sides of
\[
(H\otimes H\otimes \Delta)(\phi)(\Delta\otimes H\otimes H)(\phi)=
(1_H\otimes \phi)(H\otimes\Delta\otimes H)(\phi)(\phi\otimes 1_H)
\]
and using that $\varepsilon(\phi^1)\varepsilon(\phi^2)\varepsilon(\phi^3)$ is invertible in $\Bbbk$
we obtain $\varepsilon(\phi^1)\varepsilon(\phi^2)\varepsilon(\phi^3)=1$. Correlated to
$\varepsilon(\phi^1)\varepsilon(\phi^2)\varepsilon(\phi^3)=\varepsilon(\rho)\varepsilon(\lambda^{-1})$
this yields $\varepsilon(\lambda)=\varepsilon(\rho)$.

Let us denote $c:=\varepsilon(\lambda)=\varepsilon(\rho)$. If $u=c^{-1}\rho\otimes \lambda$
then one can see easily that $u$ is invertible and $\lambda_u=\rho_u=1_H$. Thus $H_u$ is a quasi-bialgebra
for which $\lambda_u=\rho_u=1_H$. Furthermore, $\Gamma(u): {}_H\mathfrak{M}\rightarrow {}_{H_u}\mathfrak{M}$
is a monoidal category isomorphism and $\mathfrak{U}_u\circ \Gamma(u)=\mathfrak{U}$, where we denoted by
$\mathfrak{U}_u: {}_{H_u}\mathfrak{M}\rightarrow \mathfrak{M}$ the corresponding forgetful functor.
Since a composition of monoidal functors is monoidal as well, we get that $\mathfrak{U}$ is monoidal
if and only if $\mathfrak{U}_u$ is so. So we reduced the problem to the case when $\lambda=\rho=1_H$, as desired.

Assume now that $(\mathfrak{U}, \phi_0, \phi_2): {}_H\mathfrak{M}\rightarrow \mathfrak{M}$ is a monoidal functor, and
take $v=\phi_2(H, H)(1_H\otimes 1_H)\in H\otimes H$. We first show that $v$ is invertible, and that it
determines completely $\phi_2$. To this end, let $X, Y$ be two left $H$-modules and fix $x\in X$ and $y\in Y$. If
$\varphi_x: H\rightarrow X$ and $\varphi_y: H\rightarrow Y$ are the left $H$-module morphisms defined in the proof of
\thref{reconstructionQuasi} then by the naturalness of $\phi_2$ we obtain
\[
(\varphi_x\otimes \varphi_y)\phi_2(H, H)=\phi_2(X, Y)(\varphi_x\otimes \varphi_y).
\]
Evaluating the above equality in $1_H\otimes 1_H$ we get $\phi_2(X, Y)(x\otimes y)=v^1x\otimes v^2y$, where
$v=v^1\otimes v^2\in H\otimes H$.

By similar arguments applied now to $\phi_2^{-1}$, the inverse of $\phi_2$, we
deduce that $w=\phi_2^{-1}(H, H)(1_H\otimes 1_H)\in H\otimes H$ determines completely $\phi_2^{-1}$. More exactly,
for $X, Y\in {}_H\mathfrak{M}$ we have $\phi^{-1}_2(X, Y)(x\otimes y)=w^1x\otimes w^2y$, for all $x\in X$ and $y\in Y$,
where $w=w^1\otimes w^2$. Since $\phi_2$ and $\phi_2^{-1}$ are inverses it follows now that $v$ is invertible
with $v^{-1}=w$.

Finally, the commutativity of the diagram involving $\phi_2$ in the definition of a monoidal functor comes out as
\[
(H\otimes \Delta)(v)\cdot(1_H\otimes v)\cdot (x\otimes y\otimes z)=\phi\cdot(\Delta\otimes H)(v)\cdot(v\otimes 1_H)\cdot (x\otimes y\otimes z),
\]
for all $X, Y, Z\in {}_H\mathfrak{M}$ and $x\in X$, $y\in Y$ and $z\in Z$. Clearly this is equivalent to the fact that
$\phi_{v^{-1}}=1_H\otimes 1_H\otimes 1_H$.

Likewise, if $\phi_0: \Bbbk\rightarrow \Bbbk$ is a $\Bbbk$-linear isomorphism then there is a nonzero
$\varsigma\in \Bbbk$ such that $\phi_0(\kappa)=\varsigma\kappa$, for all $\kappa\in \Bbbk$. Consequently,
the required equations for $\phi_0$ in the definition of a monoidal functor are equivalent to
$\varsigma \varepsilon(v^1)v^2=1_H=\varsigma\varepsilon(v^2)v^1$, and so
$\lambda_{\varsigma^{-1}v^{-1}}=\rho_{\varsigma^{-1}v^{-1}}=1_H$. Since $\phi_{\varsigma^{-1}v^{-1}}=
\phi_{v^{-1}}=1_H\otimes 1_H\otimes 1_H$ we conclude that for the invertible element
$\mathfrak{f}=\varsigma^{-1}v^{-1}$ of $H\otimes H$ the quasi-bialgebra $H_\mathfrak{f}$ is actually an ordinary
bialgebra, as needed.

The converse is immediate since if $\mathfrak{f}\in H\otimes H$ is invertible such that
$H_\mathfrak{f}$ is an ordinary bialgebra then, according to \coref{reconstructionHopf}, the
forgetful functor $\mathfrak{U}_\mathfrak{f}: {}_{H_\mathfrak{f}}\mathfrak{M}\rightarrow \mathfrak{M}$
is monoidal. Corroborated to the fact that $\Gamma(f): {}_H\mathfrak{M}\rightarrow {}_{H_\mathfrak{f}}\mathfrak{M}$
is a monoidal category isomorphism such that $\mathfrak{U}_\mathfrak{f}\circ \Gamma(\mathfrak{f})=\mathfrak{U}$
this leads us to the desired conclusion. So our proof is finished.
\end{proof}

We shall apply the above results to the group algebra $\Bbbk[G]$.
By \thref{reconstructionQuasi}, the monoidal structures on
${}_{\Bbbk[G]}\mathfrak{M}$ induced by that of $\mathfrak{M}$ are
given by the quasi-bialgebra structures on
$\Bbbk[G]$. We next see that these structures are built on the ordinary
bialgebra structure of the group algebra $\Bbbk[G]$, providing that $G$ is torsion-free
and abelian.

\begin{lemma}\lelabel{QuasiStrGroup}
Let $G$ be a torsion-free abelian group and $\Bbbk[G]$ the group algebra over the field $\Bbbk$ associated to it,
and endowed with the ordinary bialgebra structure, that is, endowed with the coalgebra structure given by
\[
\Delta(g)=g\otimes g,~\varepsilon(g)=1,
\]
for all $g\in G$, extended by linearity and as algebra morphisms.
Suppose that the group algebra $\Bbbk[G]$ admits a quasi-bialgebra structure given by
the comultplication $\widetilde{\Delta}$, counit $\widetilde{\varepsilon}$ and elements
$\phi$, $\lambda$ and $\rho$. Then $(\Bbbk[G],\widetilde{\Delta},\widetilde{\varepsilon})$ is a bialgebra
isomorphic to the ordinary bialgebra structure of $\Bbbk[G]$.
\end{lemma}

\begin{proof}
First note that, by \cite[Corollary 2.5]{Gilmer-Heitmann}, if $G$ is a
torsion-free abelian group then the invertible elements in $\Bbbk \left[ G%
\right] $ are exactly those of the form $qh$ where $q\in \Bbbk \backslash
\left\{ 0\right\} $ and $h\in G.$ Now, since
\[
\Bbbk[G] \otimes \Bbbk[G]\ni h\otimes g\mapsto (h, g)\in \Bbbk[G\times G],
\]
extended by linearity, defines a bialgebra isomorphism and $G\times G$ is a
torsion-free abelian group as well, we deduce that the invertible elements in
$\Bbbk[G]\otimes \Bbbk[G]$ are of the form $qh\otimes g$ with
$q\in \Bbbk \backslash \left\{ 0\right\}$ and $h, g\in G$.

Suppose now that the group algebra $\Bbbk[G]$ admits a quasi-bialgebra structure given by
the comultplication $\widetilde{\Delta}$, counit $\widetilde{\varepsilon}$ and elements
$\phi$, $\lambda$ and $\rho$.
Denote $\Bbbk[G]$ with this quasi-bialgebra structure by $\widetilde{\Bbbk[G]}$.
Since $\widetilde{\Delta}$ is an algebra map, we get that $\widetilde{\Delta}\left(g\right)$ is invertible
in $\widetilde{\Bbbk[G]}\otimes \widetilde{\Bbbk[G]}$, so we can write $\widetilde{\Delta} \left(g\right)=q_gx\otimes y$
for some $q_g\in \Bbbk\backslash \left\{ 0\right\}$ and $x, y\in G$. From
\begin{equation*}
\left(\widetilde{\varepsilon}\otimes \widetilde{\Bbbk[G]}\right)\widetilde{\Delta}\left( h\right) =\lambda
^{-1}h\lambda ,\qquad \left(\widetilde{\Bbbk[G]}\otimes \widetilde{\varepsilon} \right)\widetilde{\Delta} \left(
h\right) =\rho ^{-1}h\rho ,\text{ for every }h\in \widetilde{\Bbbk[G]}
\end{equation*}%
and $\widetilde{\Bbbk[G]}$ commutative, we have that
\begin{equation*}
\left(\widetilde{\varepsilon} \otimes \widetilde{\Bbbk[G]}\right)\widetilde{\Delta} \left( h\right) =h,\qquad \left(
\widetilde{\Bbbk[G]}\otimes \widetilde{\varepsilon}\right)\widetilde{\Delta}\left( h\right) =h,\text{ for every }%
h\in \widetilde{\Bbbk[G]}.
\end{equation*}

Thus $\Delta $ is counital and we have $g=q_{g}\widetilde{\varepsilon} \left(x\right)y$, and hence
$y=q_g^{-1}\widetilde{\varepsilon}(x^{-1})g$. Similarly, from $g=q_{g}\widetilde{\varepsilon} \left(y\right)x$
we get $x=q_g^{-1}\widetilde{\varepsilon}(y^{-1})g$.
Summing up we deduce that
$\widetilde{\Delta}\left(g\right) =(q_{g}\widetilde{\varepsilon}(x)\widetilde{\varepsilon}(y))^{-1}g\otimes g
=\widetilde{\varepsilon}(g^{-1})g\otimes g$, for all $g\in G$. It is clear now that
$\widetilde{\Bbbk[G]}$ has actually an ordinary bialgebra structure, and that $\widetilde{\Bbbk[G]}\ni g\mapsto
\widetilde{\varepsilon}(g)g\in \Bbbk[G]$, extended by linearity, is a bialgebra isomorphism. So we are done.
\end{proof}

So if $G$ is a torsion-free abelian group then, up to an isomorphism, the quasi-bialgebra structures
on the group algebra $\Bbbk[G]$ are built on the ordinary
bialgebra structure (only $\phi,\lambda,\rho$ can be non-trivial) of $\Bbbk\left[G\right]$,
by considering it as a quasi-bialgebra via a so called Harrison $3$-cocycle, see for instance \cite{bct}.
Thus our problem reduces to the computation of $H^3_{\rm Harr}(\Bbbk\left[G\right], \Bbbk, \mathbb{G}_m)$. In the
sequel we will prove that this cohomology group is trivial. Actually, we will prove that
$H^n_{\rm Harr}(\Bbbk\left[G\right], \Bbbk, \mathbb{G}_m)$ is trivial for all $n\geq 2$ and $G$
a torsion free-abelian group.

First, we recall from \cite[\& 9.2]{Caenepeel98} the definition of the Harrison cohomology over a commutative
bialgebra over a field.

Let $H$ be a commutative $\Bbbk$-bialgebra and for $n\in \mathbb{N}$
denote by $H^{\otimes n}$ the tensor product over
$\Bbbk$ of $n$ copies of $H$. By convention $H^{\otimes 0}=\Bbbk$. For a fixed $n\in \mathbb{N}$ define the maps
$f_0, \cdots, f_{n+1}: H^{\otimes n}\rightarrow H^{\otimes n+1}$ given, for all $h_1, \cdots, h_n\in H$, by
\begin{eqnarray*}
&&f_0(h_1\otimes \cdots \otimes h_n)=1_H\otimes h_1\otimes \cdots \otimes h_n, \\
&&f_i(h_1\otimes \cdots \otimes h_i)=h_1\otimes \cdots h_{i-1}\otimes \Delta(h_i)\otimes h_{i+1}\otimes \cdots \otimes h_n,~
\mbox{for $i=1, \cdots, n$},\\
&&f_{n+1}(h_1\otimes \cdots \otimes h_n)=h_1\otimes \cdots \otimes h_n\otimes 1_H.
\end{eqnarray*}

Let $\mathbb{G}_m$ be the functor from the category of commutative $\Bbbk$-algebras to the category of
abelian groups that maps a commutative $\Bbbk$-algebra $A$ to its group of units, $\mathbb{G}_m(A)$.
At the level of morphisms $\mathbb{G}_m$ sends an algebra map $f: A\rightarrow B$ to its restriction and corestriction
at $\mathbb{G}_m(A)$ and $\mathbb{G}_m(B)$, respectively.

If we set $\delta_n=\prod\limits_{i=0}^{n+1}\mathbb{G}_m(f_i)^{(-1)^i}$, for all $n\in \mathbb{N}$, then we get a complex
\[
1\rightarrow \mathbb{G}_m(\Bbbk)\stackrel{\delta_0}{\rightarrow}\mathbb{G}_m(H)
\stackrel{\delta_1}{\rightarrow}\mathbb{G}_m(H^{\otimes 2})\stackrel{\delta_2}{\rightarrow}\cdots
\]
The cohomology groups associated to this complex are denoted by $H^n_{\rm Harr}(H, \Bbbk, \mathbb{G}_m)$, $n\geq 1$.
$H^n_{\rm Harr}(H, \Bbbk, \mathbb{G}_m)$ is called the $n$th Harrison cohomology group of $H$ with values in $\mathbb{G}_m$.

\begin{proposition}\prlabel{HarrisonTorsionFree}
Let $G$ be a torsion-free abelian group and $\Bbbk[G]$ the group algebra of $G$, endowed with the ordinary
bialgebra structure. Then $H^0_{\rm Harr}(\Bbbk[G], \Bbbk, \mathbb{G}_m)=\Bbbk\backslash \{0\}$,
$H^1_{\rm Harr}(\Bbbk[G], \Bbbk, \mathbb{G}_m)=G$ and $H^n_{\rm Harr}(\Bbbk[G], \Bbbk, \mathbb{G}_m)=1$, for
all $n\geq 2$.
\end{proposition}

\begin{proof}
As in the proof of \leref{QuasiStrGroup}, inductively,
we obtain that the units of $\Bbbk[G]^{\otimes n}$ are of the form
$qx_1\otimes \cdots \otimes x_n$, for a certain
$q\in \Bbbk \backslash \left\{ 0\right\}$ and $x_1,\cdots, x_n\in G$.

Therefore, the complex that defines the Harrison cohomology of $\Bbbk[G]$ with
coefficients in $\mathbb{G}_m$ has
as spaces $\mathbb{G}_m(\Bbbk)=\Bbbk\backslash \{0\}$ and
$\mathbb{G}_m(\Bbbk[G]^{\otimes n})=(\Bbbk\backslash \{0\})\mathbb{G}^{\otimes n}$, $n\geq 1$, and boundary morphisms
given by
\begin{eqnarray*}
&&\delta_0(q)=1,~\delta_1(qh)=q1\otimes 1,~\delta_2(qh\otimes g)=h^{-1}\otimes 1\otimes g,\\
&&\delta_{2n}(qx_1\otimes x_2\otimes \cdots \otimes x_{2n})=x_1^{-1}\otimes 1\otimes
x_2x_3^{-1}\otimes 1\otimes \cdots \otimes x_{2n-2}x_{2n-1}^{-1}\otimes 1\otimes x_{2n},~\mbox{for $n\geq 2$},\\
&&\delta_{2n+1}(qx_1\otimes x_2\otimes \cdots\otimes x_{2n+1})=q1\otimes x_2\otimes x_2\otimes x_4\otimes x_4
\otimes \cdots \otimes x_{2n}\otimes x_{2n}\otimes 1,~\mbox{for $n\geq 1$},
\end{eqnarray*}
for all $q\in \Bbbk\backslash \{0\}$ and $h, g, x_1, \cdots, x_{2n+1}\in G$. We leave the verification
of these details to the reader. Notice only that we considered $G$
written multiplicatively and denoted by $1$ its neutral element.

Hence, the kernels and the images of the morphisms $\delta_i$, $i\geq 0$, are
\[
{\rm Ker}(\delta_0)=\Bbbk\backslash \{0\},~{\rm Im}(\delta_0)=1,~{\rm Ker}(\delta_1)=G,~
{\rm Ker}(\delta_2)={\rm Im}(\delta_1)=\Bbbk 1\otimes 1,
\]
and, for $n\geq 1$,
\begin{eqnarray*}
&&{\rm Ker}(\delta_{2n})={\rm Im}(\delta_{2n-1})=\{q1\otimes x_2\otimes x_2\otimes \cdots \otimes
x_{2n-2}\otimes x_{2n-2}\otimes 1\mid q\in \Bbbk\backslash \{0\}, x_{2i}\in G\},\\
&&{\rm Ker}(\delta_{2n+1})={\rm Im}(\delta_{2n})=\{x_1\otimes 1\otimes x_3\otimes \cdots
\otimes 1\otimes x_{2n+1}\mid x_{2i+1}\in G\}.
\end{eqnarray*}
From here we conclude that
$H^n_{\rm Harr}(\Bbbk[G], \Bbbk, \mathbb{G}_m)={\rm Ker}(\delta_n)/{\rm Im}(\delta_{n-1})=
\left\{\begin{array}{cl}
\Bbbk\backslash\{0\}&\mbox{if $n=0$}\\
G&\mbox{if $n=1$}\\
1&\mbox{if $n\geq 2$}
\end{array}\right.
$, as claimed.
\end{proof}

The computations made in the proof of \prref{HarrisonTorsionFree} allow us to determine
all the quasi-bialgebra structures on a group algebra of a torsion-free abelian group.

\begin{corollary}\label{ExplStrMon}
Let $\Bbbk$ be a field and $G$ a torsion-free abelian group. Then, up to an isomorphism,
to give a quasi-bialgebra structure on the group algebra $\Bbbk[G]$ is equivalent to give
an element of $(\Bbbk\backslash \{0\})\times G\times G$. More exactly, a quasi-bialgebra
structure on $\Bbbk[G]$ is given by the ordinary bialgebra structure of $\Bbbk[G]$,
\begin{equation}\eqlabel{QuasiStr}
\phi^{h, g}=h\otimes 1\otimes g,~\lambda=qg^{-1}~\mbox{and}~\rho=qh,
\end{equation}
for a certain triple $(q, h, g)\in (\Bbbk \backslash \{0\})\times G\times G$.
Furthermore, the only ordinary quasi-bialgebra structure that can be built on the ordinary
bialgebra structure of $\Bbbk[G]$ is the trivial one, in the sense
that it coincides with the ordinary bialgebra structure of $\Bbbk[G]$.
\end{corollary}
\begin{proof}
By the comments made after \leref{QuasiStrGroup} we have that, up to an isomorphism,
the quasi-bialgebra structures built on the group algebra $\Bbbk[G]$ are in a one to one correspondence with
the Harrison $3$-cocycles on the ordinary bialgebra $\Bbbk[G]$,
with coefficients in $\mathbb{G}_m$. Since an element
of ${\rm Ker}(\delta_3)$ is of the form $h\otimes 1\otimes g$, for some
$h, g\in G$, we deduce that the desired quasi-bialgebra structures are completely determined by
$\phi^{h, g}:=h\otimes 1\otimes g$ and $\lambda, \rho\in (\Bbbk\backslash \{0\})G$ such that
$(H\otimes \varepsilon\otimes H)(\phi^{h, g})=\rho\otimes \lambda^{-1}$. The latest condition
is clearly equivalent to the existence of a nonzero scalar $q$ such that $\rho=qh$ and $\lambda=qg^{-1}$.
The converse is obvious: for any triple $(q, h, g)\in (\Bbbk \backslash \{0\})\times G\times G$
and $\phi^{h, g}$, $\lambda$ and $\rho$ as in \equref{QuasiStr} we have that
$(\Bbbk[G], \phi^{h, g}, \lambda, \rho)$ is a quasi-bialgebra.

Now, the ordinary quasi-bialgebra structures built on the algebra
structure of $\Bbbk[G]$ are those for which $\rho=\lambda=1$. This forces $q=1$ and
$h=g=1$, and therefore we land at the ordinary bialgebra structure of $\Bbbk[G]$.
\end{proof}

Observe that for $\Bbbk[G]^{h, g}_q:=(\Bbbk[G], \phi^{h, g}, \lambda, \rho)$ as in \equref{QuasiStr}
the invertible element $u$ that deforms this quasi-bialgebra structure in an ordinary one is
$u=qh\otimes g^{-1}$. A simple inspection shows that $(\Bbbk[G]^{h, g}_q)_u=\Bbbk[G]$, and so
$\Bbbk[G]^{h, g}_q=\Bbbk[G]_{u^{-1}}$ is a deformation by an invertible element of the ordinary
bialgebra structure of $\Bbbk[G]$. In particular this implies the following result.

\begin{corollary}\label{MonCatEquiv}
Let $G$ be a torsion-free abelian group and $\Bbbk[G]$ the group algebra over $\Bbbk$ associated to
$G$. Then any monoidal structure on the category ${}_{\Bbbk[G]}\mathfrak{M}$ induced by that of $\mathfrak{M}$
is monoidal isomorphic to the strict monoidal category of left representations over the ordinary bialgebra $\Bbbk[G]$.
\end{corollary}
\begin{proof}
By \thref{reconstructionQuasi}, any  monoidal structure on the category
${}_{\Bbbk[G]}\mathfrak{M}$ induced by that of $\mathfrak{M}$ is determined by a quasi-bialgebra structure on $\Bbbk[G]$.
By Corollary \ref{ExplStrMon}, up to isomorphism, these structures are of the form $\Bbbk[G]^{h, g}_q$ as above.
Since $(\Bbbk[G]^{h, g}_q)_u=\Bbbk[G]$ for $u=qh\otimes g^{-1}$, we get a monoidal category isomorphism $%
\Gamma \left( u \right) :{_{\Bbbk[G]^{h, g}_q}}\mathfrak{M}\rightarrow {_{\Bbbk[G]}}\mathfrak{M}$.
\end{proof}

A first example of torsion-free abelian group is $\mathbb{Z}$, the group of integers.
In the following we will adopt the multiplicative notation $\left\langle
g\right\rangle $ for the group $\mathbb{Z}$, where $g$ is a generator.

\begin{theorem}\thlabel{teo:LaurentQuasi}
Let $\left( \Bbbk [\left\langle g\right\rangle],\Delta ,\varepsilon ,\phi ,\lambda ,\rho \right) $
be a quasi-bialgebra structure on the group algebra $\Bbbk [\left\langle g\right\rangle]$. Then,
up to an isomorphism, the quasi-bialgebra structure of $\Bbbk [\langle g\rangle]$ is completely determined by
some fixed elements $q\in \Bbbk \backslash \left\{ 0\right\}$ and $a, b\in \mathbb{Z}$,
in the sense that
\begin{eqnarray*}
\Delta \left( g\right) &=&g\otimes g,\qquad \varepsilon \left( g\right)=1, \\
\phi &=&g^{a}\otimes 1_{H}\otimes g^{b}, \\
\lambda &=&qg^{-b},\qquad \rho =qg^{a}.
\end{eqnarray*}
Furthermore, if we denote this quasi-bialgebra structure on $\Bbbk[\langle g\rangle]$
by $\Bbbk[\langle g\rangle]^{a, b}_q$ then
$\Bbbk[\langle g\rangle]^{a, b}_q=\Bbbk[\langle g\rangle]_{q^{-1}g^{-a}\otimes g^{b}}$.
Consequently, up to a monoidal category isomorphism, there is only one monoidal
structure on the category of left representations over the group algebra $\Bbbk[\langle g\rangle]$
that is induced by the strict monoidal structure of $\mathfrak{M}$. Namely, the one corresponding to the
ordinary bialgebra structure of $\Bbbk[\langle g\rangle]$.
\end{theorem}

\begin{proof}
Follows from Corollaries \ref{ExplStrMon} and \ref{MonCatEquiv}, specialized for the case $G=\mathbb{Z}$.
\end{proof}

We move now to the quasi-triangular case. We prove that there is exactly one
braided monoidal structure (actually symmetric) on the category of representation of
a group algebra associated to a torsion-free abelian group.

\begin{proposition}\prlabel{pro:triang}
Let $G$ be a torsion-free abelian group, $q\in \Bbbk\backslash \{0\}$ and $h, g\in G$.
If $\Bbbk[G]^{h, g}_q$ is the group algebra $\Bbbk[G]$ equipped with the quasi-bialgebra
structure from Corollary \ref{ExplStrMon} then $R^{h, g}=gh\otimes (gh)^{-1}$ is the only
matrix that makes $\Bbbk[G]^{h, g}_q$ a quasi-triangular (actually triangular) quasi-bialgebra.
Moreover, $(\Bbbk[G]^{h, g}_q, R^{h, g})=(\Bbbk[G], 1\otimes 1)_{q^{-1}h^{-1}\otimes g}$,
as triangular quasi-bialgebras.
\end{proposition}

\begin{proof}
If $u=qh\otimes g^{-1}$ then we have seen that $\Bbbk[G]^{h, g}_q=\Bbbk[G]_{u^{-1}}$. Thus, if
$\widetilde{R}\in \Bbbk[G]^{h, g}_q\otimes \Bbbk[G]^{h, g}_q$ endows $\Bbbk[G]^{h, g}_q$ with a
quasi-triangular structure then $\widetilde{R}_u$ is an $R$-matrix for
$(\Bbbk[G]^{h, g}_q)_u=\Bbbk[G]$. Likewise, if $R$ is an $R$-matrix on $\Bbbk[G]$ then
$R_{u^{-1}}$ defines a quasi-triangular structure on $\Bbbk[G]_{u^{-1}}=\Bbbk[G]^{h, g}_q$.
So we have to compute the quasi-triangular structures of the ordinary bialgbra $\Bbbk[G]$.

The definition of a quasi-triangular bialgebra can be obtained from that of a quasi-bialgebra
by considering $\phi=1\otimes 1\otimes 1$. So we are looking for an
invertible element $R\in \Bbbk[G]\otimes \Bbbk[G]$ such that (\ref{form:quasi3}) holds and
\[
(\Delta\otimes \Bbbk[G])(R)=(R^1\otimes 1\otimes R^2)(1\otimes R^1\otimes R^2)~,~
(\Bbbk[G]\otimes \Delta)(R)=(R^1\otimes 1\otimes R^2)(R^1\otimes R^2\otimes 1).
\]
Note that, since $\Bbbk[G]$ is both commutative and cocommutative,
equation (\ref{form:quasi3}) is always true. Since $R$ is invertible, it is of the form
$R=tx\otimes y$ for some $t\in \Bbbk \backslash \left\{ 0\right\}$
and $x, y\in G$. So we have
\[
tx\otimes x\otimes y=t^2x\otimes x\otimes y^2~\mbox{and}~tx\otimes y\otimes y=t^2x^2\otimes y\otimes y.
\]
The above equalities imply $R=1\otimes 1$, thus the bialgebra $\Bbbk[G]$ admits only one $R$-matrix,
the trivial one. From the above we get that $\Bbbk[G]^{h, g}_q$ has a unique quasi-triangular
(actually triangular) structure given by
\[
R^{h, g}:=R_{u^{-1}}=(q^{-1}g\otimes h^{-1})(qh\otimes g^{-1})=gh\otimes (gh)^{-1}.
\]
It is clear that $(\Bbbk[G]^{h, g}_q, R^{h, g})=(\Bbbk[G], 1\otimes 1)_{q^{-1}h^{-1}\otimes g}$,
as triangular quasi-bialgebras, and this finishes the proof.
\end{proof}

\begin{notation}
Consider the quasi-bialgebra $\Bbbk[G]_{q}^{h, g}$. In view of \prref{pro:triang},
there is a unique element $R$, namely $R^{h, g}=gh\otimes (gh)^{-1}$, such that
$\left(\Bbbk[G] _{q}^{h, g}, R\right)$ is a quasi-triangular
(in fact triangular) quasi-bialgebra. By abuse of notation, the datum $\left(
\Bbbk[G] _{q}^{h ,g},R^{h, g}\right)$ will be simply denoted by
$\Bbbk[G]_{q}^{h ,g}$.
\end{notation}

From the braided monoidal categorical point of view, up to isomorphism,
$\Bbbk[G]_{q}^{h ,g}$ is the ``unique" (quasi)triangular
quasi-bialgebra structure that can be built on the group algebra $\Bbbk[G]$, in the
case when $G$ is a torsion-free abelian group.

\begin{corollary}\label{coro:kgequiv}
Let $G$ be a torsion-free abelian group. Then, up to a braided monoidal category isomorphism,
we have a unique braided monoidal structure (actually symmetric) on the category of
representations over the group algebra $\Bbbk[G]$, considered monoidal via a structure induced by that of
$\mathfrak{M}$. Namely, the one induced by the trivial (quasi)triangular
structure of the ordinary bialgebra $\Bbbk[G]$.
\end{corollary}

If we take $G=\mathbb{Z}$ and keep the notations as in the statement of \thref{teo:LaurentQuasi} we then get
the following.

\begin{corollary}\colabel{BraidedZ}
Up to isomorphism, to give a (quasi)triangular quasi-bialgebra structures on the group algebra
$\Bbbk[\langle g\rangle]$ is equivalent to give a triple
$(q, a, b)\in (\Bbbk\backslash \{0\})\times \mathbb{Z}\times \mathbb{Z}$. For such a triple
$(q, a, b)$ we have a unique (quasi)triangular quasi-bialgebra structure on the group algebra
$\Bbbk[\langle g\rangle]$, and this is $\Bbbk[\langle g\rangle]^{a, b}_q$ equipped with the
$R$-matrix $R^{a, b}:=g^{a+b}\otimes g^{-a-b}$. Furthermore,
$(\Bbbk[\langle g\rangle]^{a, b}_q, R^{a, b})=(\Bbbk[\langle g\rangle], 1\otimes 1)_{q^{-1}g^{-a}\otimes g^b}$,
and so, up to a braided monoidal category isomorphism, the category ${}_{\Bbbk[\langle g\rangle]}\mathfrak{M}$
admits a unique braided monoidal (actually symmetric) structure, if it is considered monoidal via a structure
induced by that of $\mathfrak{M}$. Namely, the one obtained from the trivial (quasi)triangular structure
of the ordinary bialgebra $\Bbbk[\langle g\rangle]$.
\end{corollary}

\section{The Hom-category}

Let $G$ be a free abelian group.
It can be shown that a representation of the group algebra $\Bbbk[G]$ identifies with a pair
$(M,(f_{g})_{g\in S})$, where $M$ is a $\Bbbk$-vector space and $(f_{g})_{g\in S}$ is a family of commuting
$\Bbbk$-automorphisms of $M$ indexed by a set of generators of $G$. This gives us a new description
of the category ${}_{\Bbbk[G]}\mathfrak{M}$. In the case when $G=\langle g\rangle$ is the infinite cyclic group
we will see that this description of ${}_{\Bbbk[\langle g\rangle]}\mathfrak{M}$ coincides with a so called
Hom-category, previously introduced in \cite[Section 1]{CG}. This will allow us to describe, up to an isomorphism,
all the braided monoidal structure on the Hom-category of $\mathfrak{M}$.

\begin{definition}
Let $\mathcal{C}$ be an ordinary category. We associate to $\mathcal{C}$ a new category $%
\mathcal{H}\left( \mathcal{C}\right) $ as follows.
Objects are pairs $\left( M,f_{M}\right) $ with $M\in \mathcal{C}$ and $%
f_{M}\in \mathrm{Aut}_{\mathcal{C}}(M).$ A morphism $\xi :\left(
M,f_{M}\right) \rightarrow \left( N,f_{N}\right) $ is a morphism $\xi
:M\rightarrow N$ in $\mathcal{C}$ such that%
\begin{equation}
f_{N}\circ \xi =\xi \circ f_{M}.  \label{form:MorphH}
\end{equation}%
The category $\mathcal{H}\left( \mathcal{C}\right) $ is called the\emph{\
Hom-category} associated to $\mathcal{C}$.
\end{definition}

In the case when $\mathcal{C}=\mathfrak{M}$ we have the following description
for $\mathcal{H}\left( \mathcal{C}\right)$.

\begin{proposition}\prlabel{LaurentHom}
\label{pro:W}
We have a category isomorphism $W:{_{\Bbbk [\langle
g\rangle]}}\mathfrak{M}\rightarrow \mathcal{H}\left( \mathfrak{M}%
\right) $, given on objects by%
\begin{equation*}
W\left( X,\mu _{X}:\Bbbk [\left\langle g\right\rangle] \otimes X\rightarrow
X\right) =\left( X,f_{X}:X\rightarrow X\right) ,
\end{equation*}%
where $f_{X}\left( x\right) :=\mu _{X}\left( g\otimes x\right) ,$ for all $x\in X$,
and on morphisms by $W\xi =\xi $.
\end{proposition}

\begin{proof}
It can be easily seen that
\begin{equation*}
\mathrm{Alg}_{\Bbbk }\left( \Bbbk [\left\langle g\right\rangle], \mathrm{End}%
_{\Bbbk }\left( V\right) \right) \cong \mathrm{Grp}\left( \left\langle
g\right\rangle ,\mathrm{Aut}_{\Bbbk }\left( V\right) \right) \cong \mathrm{%
Aut}_{\Bbbk }\left( V\right) .
\end{equation*}%
Therefore, to a left $\Bbbk [\left\langle
g\right\rangle]$-module $\left( V,\mu _{V}\right) $ corresponds a pair $%
\left( V,f_{V}\right)$, where $V$ is a $\Bbbk $-vector space and $f_{V}\in
\mathrm{Aut}_{\Bbbk }\left( V\right) .$ The correspondence is given by%
\begin{equation*}
gv:=f_{V}\left( v\right) ,\text{ for all }v\in V.
\end{equation*}%
A morphism $\xi :\left( V,\mu _{V}\right) \rightarrow \left( W,\mu
_{W}\right) $ of left $\Bbbk [\left\langle g\right\rangle]$-modules
corresponds to a $\Bbbk $-linear map $\xi :V\rightarrow W$ such that, for
every $v\in V,\xi \left( gv\right) =g\xi \left( v\right) ,$ i.e., for every $%
v\in V$, $\xi \left( f_{V}\left( v\right) \right) =f_{W}\xi \left( v\right) $
or, equivalently, $\xi \circ f_{V}=f_{W}\circ \xi .$
\end{proof}

\begin{theorem}
\label{teo:meta}Let $\left( \mathcal{A},\otimes ,\mathbf{1},a,r,l\right) $
be a monoidal category, let $\mathcal{A}^{\prime }$ be a category and let $W:%
\mathcal{A}\rightarrow \mathcal{A}^{\prime }$ be a category isomorphism. For
every $X^{\prime },Y^{\prime },Z^{\prime }\in \mathcal{A}^{\prime }$ we set $%
X:=W^{-1}\left( X^{\prime }\right) ,Y:=W^{-1}\left( Y^{\prime }\right) $ and
$Z:=W^{-1}\left( Z^{\prime }\right) $, and define%
\begin{eqnarray*}
X^{\prime }\otimes ^{\prime }Y^{\prime } \hspace{-2mm}&:=&\hspace{-2mm}W\left( X\otimes Y\right)
,\qquad \mathbf{1}^{\prime }:=W\left( \mathbf{1}\right) , \\
l_{X^{\prime }}^{\prime } \hspace{-2mm}&:=&\hspace{-2mm}\left( \mathbf{1}^{\prime }\otimes ^{\prime
}X^{\prime }=W\left( \mathbf{1}\otimes X\right) \overset{Wl_{X}}{%
\longrightarrow }WX=X^{\prime }\right) , \\
r_{X^{\prime }}^{\prime }\hspace{-2mm} &:=&\hspace{-2mm}\left( X^{\prime }\otimes ^{\prime }\mathbf{1}%
^{\prime }=W\left( X\otimes \mathbf{1}\right) \overset{Wr_{X}}{%
\longrightarrow }WX=X^{\prime }\right) , \\
a_{X^{\prime },Y^{\prime },Z^{\prime }}^{\prime } \hspace{-2mm}&:=&\hspace{-2mm}\left( (
X^{\prime }\otimes ^{\prime }Y^{\prime }) \otimes ^{\prime }Z^{\prime
}=W( ( X\otimes Y) \otimes Z) \overset{Wa_{X,Y,Z}}{%
\longrightarrow }W( X\otimes ( Y\otimes Z))
=X^{\prime }\otimes ^{\prime }( Y^{\prime }\otimes ^{\prime }Z^{\prime
}) \right).
\end{eqnarray*}%
Then $\left( \mathcal{A}^{\prime },\otimes ^{\prime },\mathbf{1}^{\prime
},a^{\prime },r^{\prime },l^{\prime }\right) $ is monoidal. Moreover, $\left(
W,w_{0},w_{2}\right) :\left( \mathcal{A},\otimes ,\mathbf{1},a,r,l\right)
\rightarrow \left( \mathcal{A}^{\prime },\otimes ^{\prime },\mathbf{1}%
^{\prime },a^{\prime }, r^{\prime }, l^{\prime }\right) $ is a strict monoidal
isomorphism functor.

Furthermore, if $\left( \mathcal{A},\otimes ,\mathbf{1},a,r,l,c\right) $ is
(symmetric) braided then so is $\left( \mathcal{A}^{\prime },\otimes ^{\prime },%
\mathbf{1}^{\prime },a^{\prime },r^{\prime },l^{\prime },c^{\prime }\right)
, $ where%
\begin{equation*}
c_{X^{\prime },Y^{\prime }}^{\prime }=\left( X^{\prime }\otimes Y^{\prime
}=W\left( X\otimes Y\right) \overset{Wc_{X,Y}}{\longrightarrow }W\left(
Y\otimes X\right) =Y^{\prime }\otimes X^{\prime }\right),
\end{equation*}%
and via these structures $\left( W,w_{0},w_{2}\right) $ becomes a strict isomorphism of
(symmetric) braided monoidal categories.
\end{theorem}

\begin{proof}
It is straightforward, cf. \cite[4.4.3 and 4.4.5]{SR}.

\end{proof}

\begin{theorem}
\label{teo:mainVec}
Up to isomorphism, the monoidal structures on the category
$\mathcal{H}(\mathfrak{M})$ induced by the strict monoidal structure of $\mathfrak{M}$
are completely determined by triples $(q, a, b)\in
(\Bbbk\backslash\{0\})\times \mathbb{Z}\times \mathbb{Z}$. Explicitly,
if $(\mathcal{H}(\mathfrak{M}), \otimes, (\Bbbk, f_\Bbbk), l, r)$ is such a structure then
there exists $(q, a, b)\in
(\Bbbk\backslash\{0\})\times \mathbb{Z}\times \mathbb{Z}$ such that
$\mathcal{H}(\mathfrak{M})=\mathcal{H}^{a, b}_q(\mathfrak{M})$ as monoidal category, where
by $\mathcal{H}^{a, b}_q(\mathfrak{M})$ we denote the category $\mathcal{H}(\mathfrak{M})$
equipped with monoidal structure given by
\begin{eqnarray*}
&&
\left(X, f_{X}\right) \otimes \left(Y, f_{Y}\right) =\left( X\otimes
Y, f_{X}\otimes f_{Y}\right) \qquad , \qquad\left( \mathbf{\Bbbk },%
\mathrm{Id}_{\Bbbk }\right),\\
&&
a_{\left( X,f_{X}\right) ,\left( Y,f_{Y}\right) ,\left( Z,f_{Z}\right)
}\left( \left( x\otimes y\right) \otimes z\right) =f_{X}^{a}\left(
x\right) \otimes \left( y\otimes f_{Z}^{b}\left( z\right) \right),~
\mbox{for all $x\in X$, $y\in Y$, $z\in Z$,} \\
&&
l_{(X, f_X)}=l_X: (\Bbbk\otimes X, \mathrm{Id}_\Bbbk\otimes f_X)\rightarrow (X, f_X),~
l_{X}\left(\kappa\otimes x\right)
=\kappa qf_{X}^{-b}\left( x\right),~\mbox{for all $\kappa\in \Bbbk$, $x\in X$,}\\
&&
r_{(X, f_X)}=r_X: (X\otimes \Bbbk, f_X\otimes \mathrm{Id}_\Bbbk)\rightarrow (X, f_X),~
r_{X}\left( x\otimes \kappa\right) =\kappa qf_{X}^{a}\left( x\right),~
\mbox{for all $\kappa\in \Bbbk$, $x\in X$.}
\end{eqnarray*}
Moreover, the monoidal category $\mathcal{H}^{a, b}_q(\mathfrak{M})$ admits a
unique braided (actually symmetric) monoidal structure, given by the braiding
\[
c_{\left( X,f_{X}\right) ,\left( Y,f_{Y}\right) }=c_{X, Y}: (X\otimes Y, f_X\otimes f_Y)
\rightarrow (Y\otimes X, f_Y\otimes f_X),~
c_{X, Y}\left( x\otimes y\right)
=f_{Y}^{-a-b}\left( y\right) \otimes f_{X}^{a+b}\left( x\right),
\]%
for all $x\in X$ and $y\in Y$. Consequently, the functor $W$ defined in
Proposition \ref{pro:W} produces a strict symmetric
monoidal category isomorphism
\begin{equation*}
\left( W,w_{0},w_{2}\right) :{_{\Bbbk [\left\langle g\right\rangle]
_{q}^{a,b}}}\mathfrak{M}\rightarrow \mathcal{H}_{q}^{a,b}\left(
\mathfrak{M}\right) .
\end{equation*}
\end{theorem}

\begin{proof}
By Proposition \ref{pro:W} we have a category isomorphism
$W:{_{\Bbbk[\left\langle g\right\rangle]}}\mathfrak{M}\rightarrow \mathcal{H}\left(
\mathfrak{M}\right) .$ By Theorem \ref{teo:meta} the monoidal structures on
$\mathcal{H}(\mathfrak{M})$ are in a one to one correspondence with those of
${}_{\Bbbk[\langle g\rangle]}\mathfrak{M}$. So, according to \thref{reconstructionQuasi},
the monoidal structures on $\mathcal{H}\left(\mathfrak{M}\right)$ induced by the
strict monoidal structure of $\mathfrak{M}$ are given by the quasi-bialgebra structures
of $\Bbbk[\langle g\rangle]$.
Using \thref{teo:LaurentQuasi} we get that, up to isomorphism, the desired monoidal structures on
$\mathcal{H}(\mathfrak{M})$ are completely determined by triples
$(q, a, b)\in (\Bbbk\backslash \{0\})\times \mathbb{Z}\times \mathbb{Z}$
as follows.

Let $\left(X, f_{X}\right), \left(Y, f_{Y}\right), \left(Z, f_{Z}\right)$
be objects in $\mathcal{H}\left(\mathfrak{M}\right)$.
The tensor product in $\mathcal{H}\left( \mathfrak{M}%
\right) $ is then given by
\begin{equation*}
\left( X,f_{X}\right) \otimes \left( Y,f_{Y}\right) :=W\left( \left( X,\mu
_{X}\right) \otimes \left( Y,\mu _{Y}\right) \right) =W\left( X\otimes Y,\mu
_{X\otimes Y}\right) =\left( X\otimes Y,f_{X\otimes Y}\right),
\end{equation*}%
where
\begin{eqnarray*}
f_{X\otimes Y}\left(x\otimes y\right) &=&\mu _{X\otimes Y}\left( g\otimes
\left( x\otimes y\right) \right) =g\left( x\otimes y\right) =\Delta \left(
g\right) \left( x\otimes y\right) \\
&=&gx\otimes gy=f_{X}\left( x\right) \otimes f_{Y}\left( y\right) =\left(
f_{X}\otimes f_{Y}\right) \left( x\otimes y\right),
\end{eqnarray*}%
and so $\left(X, f_{X}\right)\otimes \left( Y,f_{Y}\right)=\left(
X\otimes Y, f_{X}\otimes f_{Y}\right) .$ The unit is $W\left( \mathbf{\Bbbk ,%
}\mu _{\Bbbk }\right) =\left( \mathbf{\Bbbk}, \mathrm{Id}_\Bbbk\right)$ since
\begin{equation*}
f_{\Bbbk }\left(\kappa\right) :=\mu _{\Bbbk }\left( g\otimes \kappa\right) =g\cdot
\kappa =\varepsilon \left( g\right)\kappa=\kappa,
\end{equation*}
for all $\kappa\in \Bbbk$. The left unit constraint is given, for every
$\kappa\in \Bbbk$, $x\in X$, by%
\begin{eqnarray*}
l_{\left( X,f_{X}\right) }\left(\kappa \otimes x\right)&=&\left( Wl_{\left(
X,\mu _{X}\right) }\right) \left(\kappa\otimes x\right) =l_{\left( X,\mu
_{X}\right) }\left(\kappa \otimes x\right) \\
&=&\kappa l_{\left( X,\mu _{X}\right) }\left( 1_{\Bbbk }\otimes x\right)
=\kappa \left(\lambda x\right) =\kappa \left(qg^{-b}x\right)
=\kappa qf_{X}^{-b}\left( x\right).
\end{eqnarray*}%
Likewise, the right unit constraint is given, for every $\kappa\in \Bbbk$, $x\in X$, by%
\begin{eqnarray*}
r_{_{\left( X,f_{X}\right) }}\left( x\otimes \kappa\right)&=&\left( Wr_{\left(
X,\mu _{X}\right) }\right) \left( x\otimes \kappa\right) =r_{\left( X,\mu
_{X}\right) }\left( x\otimes \kappa\right) \\
&=&\left( r_{\left( X,\mu _{X}\right) }\right) \left( x\otimes 1_{\Bbbk
}\right) \kappa=\left( \rho x\right)\kappa=\kappa\left(qg^{a}x\right)
=\kappa qf_{X}^{a}\left( x\right) .
\end{eqnarray*}

In a similar manner we compute that the associativity constraint is given, for every $x\in X$, $y\in Y$, $z\in Z$,
by
\begin{eqnarray*}
a_{\left( X,f_{X}\right) ,\left( Y,f_{Y}\right) ,\left( Z,f_{Z}\right)
}\left( \left( x\otimes y\right) \otimes z\right)&=&\left( Wa_{\left( X,\mu
_{X}\right) ,\left( Y,\mu _{Y}\right) ,\left( Z,\mu _{Z}\right) }\right)
\left( \left( x\otimes y\right) \otimes z\right)  \\
&=&a_{\left( X,\mu _{X}\right) ,\left( Y,\mu _{Y}\right) ,\left( Z,\mu
_{Z}\right) }\left( \left( x\otimes y\right) \otimes z\right) =\phi
^{1}x\otimes \left( \phi ^{2}y\otimes \phi ^{3}z\right)  \\
&=&g^{a}x\otimes \left( y\otimes g^{b}z\right)
=f_{X}^{a}\left( x\right) \otimes \left( y\otimes f_{Z}^{b}\left(z\right) \right).
\end{eqnarray*}%
Thus if we transport through $W$ the monoidal structure of ${}_{\Bbbk[\langle g\rangle]^{a, b}_p}\mathfrak{M}$
what we get on $\mathcal{H}(\mathfrak{M})$ is the monoidal structure of $\mathcal{H}(\mathfrak{M})^{a, b}_p$,
as needed.

Since ${}_{\Bbbk[\langle g\rangle]^{a, b}_p}\mathfrak{M}$ has a unique braided (actually symmetric)
monoidal structure by the above comments it follows that $\mathcal{H}(\mathfrak{M})^{a, b}_p$
has a unique braided (actually symmetric) monoidal structure, too. It is given by the braiding
defined, for every $x\in X$, $y\in Y$, by
\begin{eqnarray*}
c_{\left( X,f_{X}\right) ,\left( Y,f_{Y}\right) }\left( x\otimes y\right)
&=&\left( Wc_{\left( X,\mu _{X}\right) ,\left( Y,\mu _{Y}\right) }\right)
\left( x\otimes y\right) =c_{\left( X,\mu _{X}\right) ,\left( Y,\mu
_{Y}\right) }\left( x\otimes y\right)  \\
&=&R^{2}y\otimes R^{1}x=g^{-a-b}y\otimes g^{a+b}x=f_{Y}^{-a-b}\left( y\right)
\otimes f_{X}^{a+b}\left( x\right).
\end{eqnarray*}
The last assertion follows easily from Theorem \ref{teo:meta}.
\end{proof}

\begin{corollary}
\label{coro:Hequiv}Let $q\in \Bbbk \backslash \left\{ 0\right\}$ and
$a, b\in\mathbb{Z}$. We have isomorphisms of symmetric monoidal categories%
\begin{equation*}
\mathcal{H}_{q}^{a,b}\left( \mathfrak{M}\right)
\cong {_{\Bbbk\left\langle g\right\rangle _{q}^{a,b}}}\mathfrak{M}
\cong {}_{\Bbbk[\langle g\rangle]}\mathfrak{M}
\cong \mathcal{H}_{1}^{0, 0}\left( \mathfrak{M}\right)
\text{.}
\end{equation*}
\end{corollary}

\begin{proof}
By Theorem \ref{teo:mainVec} we have
$\mathcal{H}_{q}^{a,b}\left( \mathfrak{M}\right)
\cong {_{\Bbbk\left\langle g\right\rangle _{q}^{a,b}}}\mathfrak{M}$, and
by Corollary \ref{coro:kgequiv} that
${_{\Bbbk\left\langle g\right\rangle _{q}^{a,b}}}\mathfrak{M}
\cong {}_{\Bbbk[\langle g\rangle]}\mathfrak{M}$. Both of them are isomorphisms of
symmetric monoidal categories.
\end{proof}

\begin{definition}
\label{def:Htilde}Let $\left( \mathcal{C},\otimes ,\mathbf{1},a,l,r\right) $
be a monoidal category. Following \cite[Section 1]{CG}, the category $%
\mathcal{H}\left( \mathcal{C}\right) $ becomes a monoidal category $\mathcal{%
H}\left( \mathcal{C}\right) =\left( \mathcal{H}\left( \mathcal{C}\right)
,\otimes ,\left( \mathbf{1},Id_{\mathbf{1}}\right) ,a,l,r\right) .$ Here by
abuse of notation we denote with the same letters the constraints of $%
\mathcal{C}$ regarded as morphisms in $\mathcal{H}\left( \mathcal{C}\right) $
(thus, for instance $a_{\left( M,f_{M}\right) }$ is $a_{M}$ regarded as a
morphism in $\mathcal{H}\left( \mathcal{C}\right) $). The tensor product of $%
\left( M,f_{M}\right) $ and $\left( N,f_{N}\right) $ is given by the formula%
\begin{equation*}
\left( M,f_{M}\right) \otimes \left( N,f_{N}\right) =\left( M\otimes
N,f_{M}\otimes f_{N}\right) .
\end{equation*}%
At the level of morphisms, the tensor product is the tensor product of
morphisms.

In \cite[Proposition 1.1]{CG}, a modified version $\widetilde{\mathcal{H}}%
\left( \mathcal{C}\right) =\left( \mathcal{H}\left( \mathcal{C}\right)
,\otimes ,\left( \mathbf{1},Id_{\mathbf{1}}\right) ,\widetilde{a},\widetilde{%
l},\widetilde{r}\right)$ of the monoidal category $\mathcal{H}(\mathcal{C})$ was given.
Namely, the associativity constraint $\widetilde{a}$ is defined, for
$\left( M,f_{M}\right)$, $\left( N,f_{N}\right)$, $\left(P,f_{P}\right)
\in \mathcal{H}\left( \mathcal{C}\right)$, by the formula%
\begin{equation*}
\widetilde{a}_{\left( M,f_{M}\right) ,\left( N,f_{N}\right) ,\left(
P,f_{P}\right) }=a_{M,N,P}\circ \left( \left( f_{M}\otimes N\right) \otimes
f_{P}^{-1}\right) =\left( f_{M}\otimes \left( N\otimes f_{P}^{-1}\right)
\right) \circ a_{M,N,P},
\end{equation*}%
while the unit constraints $\widetilde{l}$ and $\widetilde{r}$ are defined by%
\begin{equation*}
\widetilde{l}_{\left( M,f_{M}\right) }=f_{M}\circ l_{M}=l_{M}\circ \left(
\mathbf{1}\otimes f_{M}\right) \qquad \text{and}\qquad \widetilde{r}_{\left(
M,f_{M}\right) }=f_{M}\circ r_{M}=r_{M}\circ \left( f_{M}\otimes \mathbf{1}%
\right) .
\end{equation*}%
Furthermore, by \cite[Proposition 1.2]{CG}, if $\left( \mathcal{C},\otimes ,%
\mathbf{1},a,l,r,c\right) $ is a braided monoidal category then so is $\widetilde{%
\mathcal{H}}\left( \mathcal{C}\right) =\left( \mathcal{H}\left( \mathcal{C}%
\right) ,\otimes ,\left( \mathbf{1},Id_{\mathbf{1}}\right) ,\widetilde{a},%
\widetilde{l},\widetilde{r},c\right) .$

Hence, by \cite[Proposition 1.7]{CG}, we deduce that $\left( \mathcal{H}%
\left( \mathcal{C}\right) ,\otimes ,\left( \mathbf{1},Id_{\mathbf{1}}\right)
,a,l,r,c\right) $ is braided as well.
\end{definition}

As a consequence of Theorem \ref{teo:mainVec}, we have an alternative
description for the symmetric monoidal categories given in Definition %
\ref{def:Htilde}, providing that $\mathcal{C}=\mathfrak{M}$.

\begin{proposition}
\label{pro:Hrev}
We have the following equalities of braided monoidal categories%
\begin{equation*}
\mathcal{H}\left( \mathfrak{M}\right) =\mathcal{H}_{1}^{0,0}\left(
\mathfrak{M}\right) \qquad \text{and}\qquad \widetilde{\mathcal{H}}\left(
\mathfrak{M}\right) =\mathcal{H}_{1}^{1,-1}\left( \mathfrak{M}\right) .
\end{equation*}
\end{proposition}

The result below gives a more conceptual proof for \cite[Proposition 1.7]{CG}, in the
particular case when $\mathcal{C}=\mathfrak{M}$.

\begin{corollary}\colabel{SGiso}
We have the following isomorphisms of symmetric monoidal categories%
\begin{equation*}
\mathcal{H}\left( \mathfrak{M}\right) \cong {_{\Bbbk \left\langle
g\right\rangle _{1}^{0,0}}}\mathfrak{M}\cong
{}_{\Bbbk[\langle g\rangle]}\mathfrak{M}\cong
{_{\Bbbk \left\langle g\right\rangle _{1}^{1,-1}}}\mathfrak{M}
\cong \widetilde{\mathcal{H}}%
\left( \mathfrak{M}\right) \text{.}
\end{equation*}
\end{corollary}
\begin{proof}
It follows by Corollary \ref{coro:Hequiv}. See also Proposition \ref{pro:Hrev}.
\end{proof}

\noindent \textbf{Acknowledgements.} We would like to thank J. G\'{o}mez-Torrecillas and S. Gelaki for
helpful discussions on the main topics of the paper.


\begin{thebibliography}{}

\bibitem[BCT]{bct}
D. Bulacu, S. Caenepeel, B. Torrecillas, \emph{The braided monoidal structures on the category of
vector spaces graded by the Klein group}. Proc. Edinb. Math. Soc., II. \textbf{54} (2011), 613-641.

\bibitem[Ca]{Caenepeel98}
S. Caenepeel, \emph{Brauer groups, Hopf algebras and Galois theory}.
K-Monographs Math. \textbf{4}, Kluwer Academic Publishers, Dordrecht, 1998.

\bibitem[CG]{CG} S. Caenepeel, I. Goyvaerts, \emph{Monoidal Hom-Hopf algebras%
}. Comm. Algebra \textbf{39} (2011), no. 6, 2216--2240.

\bibitem[Dr]{Drinfeld-QuasiHopf} V. G. Drinfeld, \emph{Quasi-Hopf algebras}.
(Russian) Algebra i Analiz \textbf{1} (1989), no. 6, 114--148; translation
in Leningrad Math. J. \textbf{1} (1990), no. 6, 1419--1457.

\bibitem[GH]{Gilmer-Heitmann} R. Gilmer; R.C. Heitmann, \emph{The group of
units of a commutative semigroup ring}. Pacific J. Math. \textbf{85} (1979),
no. 1, 49--64.

\bibitem[Ka]{Kassel} C. Kassel, \emph{Quantum groups}. Graduate Texts in
Mathematics, \textbf{155}. Springer-Verlag, New York, 1995.

\bibitem[Pa]{Passman} D. S. Passman, \emph{The algebraic structure of group
rings. Pure and Applied Mathematics}. Wiley-Interscience [John Wiley \&
Sons], New York-London-Sydney, 1977.

\bibitem[SR]{SR} N. Saavedra Rivano, \emph{Cat\'{e}gories Tannakiennes}.
Lecture Notes in Mathematics, Vol. \textbf{265}. Springer-Verlag, Berlin-New
York, 1972.
\end{thebibliography}
\end{document}